\documentclass[reqno]{amsart}

\usepackage{amsmath}
\usepackage{amsfonts}
\usepackage{amssymb,enumerate}
\usepackage{amsthm,stmaryrd}
\usepackage[all]{xy}
\usepackage[draft=false]{hyperref}
\usepackage{tikz}
\usetikzlibrary{matrix,arrows,decorations.pathmorphing}
\usepackage{comment}
\usepackage[obeyDraft,color=green!40]{todonotes}
\usepackage{xcolor}
\theoremstyle{plain}
\newtheorem{lem}{Lemma}[section]
\newtheorem{cor}[lem]{Corollary}
\newtheorem{prop}[lem]{Proposition}
\newtheorem{thm}[lem]{Theorem}

\newtheorem{intthm}{Theorem}

\theoremstyle{definition}
\newtheorem{defn}[lem]{Definition}
\newtheorem{ex}[lem]{Example}
\newtheorem{question}[lem]{Question}

\newtheorem{disc}[lem]{Remark}

\newtheorem{fact}[lem]{Fact}

\numberwithin{equation}{lem}

\renewenvironment{proof}{\vspace{1ex}\noindent{\textbf{Proof:}}\hspace{0.5em}}
{\hfill\qed\vspace{1ex}}

\newcommand{\D}{\mathcal{D}}

\newcommand{\cat}[1]{\mathcal{#1}}

\newcommand{\catd}{\cat{D}}

\newcommand{\catb}{\cat{B}}

\newcommand{\catbc}{\cat{B}_C}

\newcommand{\pd}{\operatorname{pd}}

\newcommand{\id}{\operatorname{id}}	
\newcommand{\fd}{\operatorname{fd}}

\newcommand{\cidim}{\mathrm{CI}\text{-}\!\dim}	
	
\newcommand{\cifd}{\mathrm{CI}\text{-}\!\fd}	
	
\newcommand{\gid}{\operatorname{Gid}}

\newcommand{\depth}{\operatorname{depth}}

\newcommand{\rhom}{\mathbf{R}\!\operatorname{Hom}}	
\newcommand{\lotimes}{\otimes^{\mathbf{L}}}
\newcommand{\HH}{\operatorname{H}}
\newcommand{\Hom}{\operatorname{Hom}}	

\newcommand{\spec}{\operatorname{Spec}}

\newcommand{\tor}{\operatorname{Tor}}
\newcommand{\im}{\operatorname{Im}}
\newcommand{\shift}{\mathsf{\Sigma}}

\newcommand{\cone}{\operatorname{Cone}}
\newcommand{\Ker}{\operatorname{Ker}}

\newcommand{\ideal}[1]{\mathfrak{#1}}
\newcommand{\m}{\ideal{m}}
\newcommand{\n}{\ideal{n}}
\newcommand{\p}{\ideal{p}}
\newcommand{\q}{\ideal{q}}

\newcommand{\supp}{\operatorname{Supp}}

\newcommand{\bbz}{\mathbb{Z}}
\newcommand{\bbn}{\mathbb{N}}

\newcommand{\from}{\leftarrow}
\newcommand{\xra}{\xrightarrow}
\newcommand{\xla}{\xleftarrow}

\newcommand{\vf}{\varphi}

\newcommand{\y}{\mathbf{y}}

\newcommand{\x}{\mathbf{x}}

\renewcommand{\geq}{\geqslant}
\renewcommand{\leq}{\leqslant}
\renewcommand{\ker}{\Ker}
\renewcommand{\hom}{\Hom}

\newcommand{\Ext}[4][R]{\operatorname{Ext}_{#1}^{#2}(#3,#4)}	
\newcommand{\Rhom}[3][R]{\mathbf{R}\!\operatorname{Hom}_{#1}(#2,#3)}	

\newcommand{\Otimes}[3][R]{#2\otimes_{#1}#3}
\renewcommand{\Hom}[3][R]{\operatorname{Hom}_{#1}(#2,#3)}

\newcommand{\width}{\operatorname{width}}

\renewcommand{\spec}[1][R]{\operatorname{Spec}(#1)}

\newcommand{\catdfb}{\catd^{\text{f}}_{\text{b}}}
\newcommand{\catdb}{\catd_{\text{b}}}
\renewcommand{\supp}{\operatorname{supp}}

\newcommand{\homciid}{\operatorname{CI}_{\operatorname{Hom}}\operatorname{-id}}

\begin{document}

\title{Complete Intersection Hom Injective Dimension}
\author{Sean K. Sather-Wagstaff}

\address{Department of Mathematical Sciences,
Clemson University,
O-110 Martin Hall, Box 340975, Clemson, S.C. 29634
USA}

\email{ssather@clemson.edu}

\urladdr{https://ssather.people.clemson.edu/}

\author{Jonathan P. Totushek}
\address{Mathematics and Computer Science Department,
University of Wisconsin-Superior,
Swenson Hall 3030,
Belknap and Catlin Ave,
P.O. Box 2000,
Superior, WI 54880,
USA}

\email{jtotushe@uwsuper.edu}

\urladdr{http://mcs-web.uwsuper.edu/jtotushe}

\dedicatory{}

\keywords{Bass conjecture, Bass numbers, complete intersection dimensions, injective dimension}
\subjclass[2010]{primary 13D05; secondary 13C11, 13D09, 13D22}

\begin{abstract}
We introduce and investigate a new injective version of the complete intersection dimension of Avramov, Gasharov, and Peeva. It is like the complete intersection injective dimension of Sahandi, Sharif, and Yassemi in that it is built using quasi-deformations. Ours is different, however, in that we use a Hom functor in place of a tensor product. We show that (a) this invariant characterizes the complete intersection property for local rings, (b) it fits between the classical injective dimension and the G-injective dimension of Enochs and Jenda, (c) it provides modules with Bass numbers that are bounded by polynomials, and (d) it improves a theorem of Peskine, Szpiro, and Roberts (Bass' conjecture).
\end{abstract}

\maketitle

\tableofcontents

\section{Introduction}
\label{180831:1}

\noindent
\textbf{Convention.}
In Sections~\ref{180831:1}--\ref{180807:4}, let $(R,\m,k)$ be a local ring.  

\

It is well known that $R$-modules with finite projective dimension are particularly nice, behaving like modules over regular local rings. 
Furthermore, modules of finite projective dimension have the ability to detect when the ring is regular according to the famous result of Auslander, Buchsbaum, and Serre~\cite{auslander:hdlr,serre:sldhdaedmn}.
The injective dimension has similar behavior.

The complete intersection dimension of Avramov, Gasharov and Peeva~\cite{avramov:cid} similarly identifies modules modules that behave like modules over over formal complete intersection rings. 
(The local ring $R$ is a \textit{formal complete intersection} if its $\m$-adic completion $\widehat{R}$ is of the form $\widehat{R} \cong Q/\mathfrak{a}$ where $Q$ is a regular local ring and $\mathfrak{a}$ is generated by a $Q$-regular sequence.)
For instance, the Betti numbers of such modules are bounded above by a polynomial, and these modules have the ability to detect when the ring is a formal complete intersection. 

The complete intersection dimension is defined in terms of \textit{quasi-deformations} of $R$, which are diagrams of local ring homomorphisms $R\xra{\varphi} R' \xla{\tau} Q$ such that $\varphi$ is flat, and $\tau$ is surjective with kernel generated by a $Q$-regular sequence. 
One then defines the \textit{complete intersection dimension} of a finitely generated $R$-module $M$ as 
\[
	\cidim_R(M) = \inf\{\pd_Q(M') - \pd_Q(R') \mid \text{ $R \to R' \from Q$ is a quasi-deformation} \}
\]
where $M' = R'\otimes_R M$.
Sahandi, Sharif, and Yassemi~\cite{sahandi:hfd} define the complete intersection flat dimension and complete intersection injective dimension for modules that are not necessarily finitely generated using $\fd_Q(M')$ and $\id_Q(M')$ in place of $\pd_Q(M')$. 
The complete intersection injective dimension is challenging to work with because injective dimension properties often do not behave well with respect to tensor products.

In this paper we introduce and investigate the following variation on the complete intersection injective dimension
	\[
            \homciid_R(M) = \inf\{\id_Q(J^{\varphi}) - \pd_Q(R') \mid \text{ $R\xra{\varphi} R' \from Q$ is a quasi-deformation} \}
	\]
where $J^{\varphi} = \hom{R'}{J}$ and $J$ is an injective resolution of $M$. 
We call this the \textit{complete intersection Hom injective dimension}.

Our first main result about this invariant is stated next. It shows how the complete intersection Hom injective dimension compares with the injective dimension and with Enochs and Jenda's~\cite{enochs:gipm} Gorenstein injective dimension. 

\begin{intthm}
\label{180711:1}
        For an $R$-module $M$, there are inequalities
	\[
		\operatorname{Gid}_R(M)\leq \homciid_R(M)\leq \id_R(M)
	\]
	with equality to the left of any finite value.
\end{intthm}

This is Theorem~\ref{160817:4}\eqref{160817:4b} below.
Theorem~\ref{160817:4} also includes a version of the Chouinard formula for the complete intersection Hom injective dimension. 
This yields a Bass formula when $M$ is finitely generated; see Corollary~\ref{180724:1}. 

Our next main result is contained in Theorem~\ref{160922.3}. 
It shows that the complete intersection Hom injective dimension characterizes formal complete intersection rings like the injective dimension characterizers regular rings.
The subsequent result (part of Theorem~\ref{170907:2}) gives a polynomial bound for the Bass numbers of a finitely generated module with finite complete intersection Hom injective dimension; in the language of Avramov, Iyengar, and Miller~\cite{avramov:holh} it says that such modules have finite injective complexity.

\begin{intthm}
\label{180711:2}
        The following conditions are equivalent.
        \begin{enumerate}[\rm (i)]
            \item $\homciid_R(k)<\infty$. 
            \item $R$ is a formal complete intersection.
            \item For every $R$-module $M$, we have $\homciid_R(M)<\infty$. 
        \end{enumerate}
\end{intthm}

\begin{intthm}
	\label{180711:3}
	Let $M$ be a finitely generated $R$-module with $\homciid_R(M)<\infty$. 
	Then the sequence of Bass numbers $\{\mu^i_R(M)\mid i\geq 0\}$ is bounded above by a polynomial in~$i$.
\end{intthm}

Our final main result improves Bass' conjecture~\cite{bass:ugr} as proved by Peskine, Szpiro, and Roberts~\cite{peskine:dpfcl,roberts:ti,roberts:it}; see Theorem~\ref{170915:4.2}.
It is an open question whether one can replace complete intersection Hom injective dimension with Gorenstein injective dimension in this result~\cite[Question~6]{christensen:btrmab}; see Takahashi and Yassemi~\cite{MR2231892,MR2287566} for progress on this question. 

\begin{intthm}
\label{180711:4}
	Assume that $R$ has non-zero finitely generated $R$-module $M$ such that $\homciid_R(M)<\infty$.
	Then $R$ is Cohen-Macaulay.
\end{intthm}

In addition to the contents summarized above, Section~\ref{180807:3} contains stability results needed for Section~\ref{180807:4}, and Appendix~\ref{180725:1} contains non-local technical results for use in Section~\ref{180807:2}.
Lastly, while this introduction is written in terms of modules, the body of the article is written in terms of complexes; Section~\ref{180807:1} contains relevant background information on the derived category.

\section{Background}
\label{180807:1}

Recall that $R$ is a local ring. We mostly work in the derived category $\D(R)$ of complexes of $R$-modules, indexed homologically (see e.g. \cite{gelfand:moha,hartshorne:rad}). Isomorphisms in $\D(R)$ are identified by the symbol $\simeq$. We use the same symbol to identify quasiisomorphisms between complexes, that is, chain maps that induce isomorphisms on the level of homology. 

Let $X,Y\in\D(R)$. Then $\inf(X)$ and $\sup(X)$ denote the infimum and supremum, respectively, of the set $\{i\in \bbz\mid \operatorname{H}_i(X) = 0\}$. Let $X\lotimes_R Y$ and $\rhom_R(X,Y)$ denote the left-derived tensor product and right-derived homomorphism complexes, respectively.
The \textit{depth} and \textit{width} of $X$ are defined by Foxby and Yassemi~\cite{foxby:bcfm,yassemi:wcm}~as 
\begin{align*}
	\depth_R(X) &= -\sup(\rhom_R(k,X))\\
	\operatorname{width}_R(X) &= \inf(k\lotimes_R X).
\end{align*}
	The $i^{\text{th}}$ \textit{Betti number} and the $i^{\text{th}}$ \textit{Bass number} of $X$ are, respectively,
	\begin{align*}
		\beta^R_i(X) &= \dim_k(\HH^i(\rhom_R(X,k)) = \dim_k(\HH_i(X \lotimes_R k)) \\
		\mu_R^i(X) &= \dim_k(\HH^i(\rhom_R(k,X)).
	\end{align*}
Let $E_R(k)$ denote the injective hull of $k$, and set $X^{\vee} = \rhom_R(X,E_R(k))$.

Let $\D_+(R)$ and $\D_-(R)$ denote the full subcategories of $\D(R)$ consisting of all complexes $X$ such that $\HH_i(X) = 0$ for $i \ll 0$ and $i \gg 0$, respectively. Set $\D_\text{b}(R) = \D_+(R) \cap \D_-(R)$, and let $\D^\text{f}_\text{b}(R)$ denote the full subcategory consisting of all complexes $X$ such that $\bigoplus_{i\in\bbz} \HH_i(X)$ is finitely generated.

	If $X\in \D_-(R)$, then the \textit{injective dimension} of $X$ is
	\[
		\id_R(X) = \inf\left\{n\in \bbz \left| 
		\begin{array}{l}
			X\xra{\simeq} I \text{ where } I \text{ is a complex of injective}\\
			R\text{-modules such that } I_i = 0 \text{ for all } i<-n
		\end{array}\right.
		\right\}.
	\]
	The \textit{flat dimension} ($\fd$) and \textit{projective dimension} ($\pd$) are defined similarly for $X\in \D_+(R)$.
	If $\id_R(X)<\infty$, then $X\in \D_{\text{b}}(R)$ and similarly for $\fd$ and $\pd$.

	For ease of reference we single out the following.

\begin{fact}
    \label{180505:1}
    	Let $S$ be a commutative noetherian ring (not necessarily local) with $\dim(S) < \infty$. If $F$ is a flat $R$-module, then $\pd_S(F)\leq \dim(S)<\infty$ by Raynaud and Gruson \cite[Seconde partie, Th\'{e}or\`{e}me (3.2.6)]{raynaud:cpptpm} and Jensen \cite[Proposition 6]{jensen:vl}.
	In particular if $X\in \D_\text{b}(S)$, then $\rhom_S(F,X)\in \D_\text{b}(S)$.
\end{fact}

	If $X\in \D_{+}(R)$, then the \textit{complete intersection flat dimension} of $X$ as defined by Sahandi, Sharif, Yassemi, and Sather-Wagstaff~\cite{sahandi:hfd,sather:cidfc} is
	\[
		\cifd_R(X) = \inf \{\fd_Q(X') - \pd_Q(R') \mid R \to R' \from Q \text{ is a quasi-deformation}\}
	\]
	where $X' = R'\lotimes_R X$.
	If $X\in \D^{\text{f}}_{\text{b}}(R)$, then this is the complete intersection dimension $\cidim_R(X)$ of Avramov, Gasharav, Peeva, and Sather-Wagstaff~\cite{avramov:cid,sather:cidc}.

	Following Enochs and Jenda~\cite{enochs:gipm}, we say that an $R$-module $N$ is \textit{Gorenstein injective} if there is an exact sequence of injective modules
	\[
		E = \cdots \to E_{2} \xra{\partial_{2}} E_{1} \xra{\partial_{1}} E_0 \xra{\partial_0} E_{-1} \xra{\partial_{-1}} E_{-2} \to \cdots
	\]
	such that $N \cong \ker(\partial_0)$ and such that for any injective module $I$, the complex $\hom{I}{E}$ is exact.

	If $X\in\D_-(R)$, then the \textit{Gorenstein injective dimension} of $X$ is 
	\[
		\gid_R(X) = \inf\left\{n\in \bbz \left| 
		\begin{array}{l}
			X\xra{\simeq} I \text{ where } I \text{ is a complex of Gorenstein injective}\\
			R\text{-modules such that } I_i = 0 \text{ for all } i<-n
		\end{array}\right.
		\right\}.
	\]

	The following concepts are due to Avramov, Christensen, and Foxby \cite{avramov:rhafgd, christensen:scatac, foxby:gmarm}.
	An $R$-complex $C\in \D^{\text{f}}_{\text{b}}(R)$ is \textit{semidualizing} if the natural homothety morphism $\chi^R_C : R \to \rhom_R(C,C)$ is an isomorphism in $\D(R)$.
	The \textit{Bass class} with respect to $C$ is the class $\catbc(R)$ of all $X\in \D_{\text{b}}(R)$ such that $\rhom_R(C,X) \in \D_{\text{b}}(R)$ and the evaluation morphism $\xi^{C}_X:C\lotimes_R \rhom_R(C,X) \to X$ is an isomorphism in $\D(R)$.
	It is straightforward to show that $\catbc(R)$ satisfies the following \textit{two-of-three} condition: for each distinguished triangle $X\to Y \to Z\to$ in $\catd(R)$, if two of the three complexes are in $\catbc(R)$, then so is the third.
	We say that $D \in \catdfb(R)$ is \textit{dualizing} for $R$ if it is semidualizing for $R$ and $\id_R(D) < \infty$.
	Recall that if $R$ is a homorphic image of a Gorenstein ring, e.g., if $R$ is complete, then $R$ has a dualizing complex by~\cite[V.10]{hartshorne:rad}.

	A full subcategory $T \subseteq \D(R)$ is \textit{thick} if it is closed under suspensions and summands, and it satisfies the two-of-three condition. 
	Given any collection $S$ of $R$-complexes, the \textit{thick subcategory generated by} $S$ is the smallest thick subcategory of $\D(R)$ containing $S$.

\begin{disc}\label{disc171022a}
Let $R'$ be an $R$-algebra and
let $T_1 = \operatorname{Add}(R')$ be the class of $R$-module summands of
arbitrary direct sums of copies of $R'$.
Inductively, for $n\geq 2$ let $T_n$ consist of the objects $Z\in\catd(R)$ such that
$Z$ is a retract of an object $Y\in\catd(R)$ such that there is an exact triangle 
$Y'\to Y\to Y''\to$
in $\catd(R)$ with $Y'\in T_1$ and $Y''\in T_{n-1}$.
From~\cite[Proposition~2.3]{mathew} we have $T=\cup_{i=1}^\infty T_n$, and furthermore
if $A'\to A\to A''\to$ is an exact triangle in $\catd(R)$ such that $X'\in T_a$ and $X''\in T_b$, then $X\in T_{a+b}$.
\end{disc}

\section{Homological Dimensions of Complexes}
\label{180807:2}

Recall that $R$ is a local ring. In this section we prove Theorem \ref{180711:1} from the introduction in addition to other results, e.g., a version of the Chouinard formula and Bass formula. 
We begin by defining the complete intersection Hom injective dimension for complexes; note that it recovers the definition for modules given in the introduction.

\begin{defn}
\label{180810:1}
	Let $X\in \D_{-}(R)$. The \textit{complete intersection Hom injective dimension} of $X$ is 
	\[
		\homciid_R(X) = \inf\{\id_Q(X^{\varphi}) - \pd_Q(R') \mid \text{ $R\xra{\varphi} R' \from Q$ is a quasi-deformation} \}
	\]
	where $X^{\varphi} = \rhom_R(R',X)$.
\end{defn}

\begin{disc}
\label{180817:1}
	Let $X\in \D_{-}(R)$ be such that $\homciid_R(X)<\infty$. Then $X\in \D_{\text{b}}(R)$. Indeed, by assumption there is a quasi-deformation $R\xra{\varphi} R'\from Q$ such that $\id_Q(X^{\varphi})<\infty$. In particular $X^{\varphi} \in \D_{\text{b}}(R')$, so $X\in \D_{\text{b}}(R)$ by Corollary~\ref{cor171022a}.

Furthermore $\homciid_R(X)=-\infty$ if and only if $X\simeq 0$.
Indeed, if $X\simeq 0$, then $\id_R(X)=-\infty$, so using the trivial quasi-deformation $R\to R\from R$ we have $\homciid_R(X)=-\infty$.
Conversely, assume that $\homciid_R(X)=-\infty$.
It follows that there is a quasi-deformation $R\xra{\varphi} R'\from Q$ such that
$\id_Q(X^{\varphi})=-\infty$.
Thus, we have $\Rhom{R'}X = X^{\varphi} \simeq 0$. Since $R'$ is faithfully flat over $R$, it follows from, e.g.,~\cite[Lemma 5.6]{sather:scc}
that $X\simeq 0$, as desired.
\end{disc}

The following result contains Theorem~\ref{180711:1} from the introduction 
and our version of the Chouinard formula. 
Furthermore, it shows that any quasi-deformation that detects the finiteness of $\homciid_R(X)$ can be used to calculate the actual value of $\homciid_R(X)$.
See Example~\ref{180913:1} for some strict inequalities in part~\eqref{160817:4b}.

\begin{thm}
\label{160817:4}
        Let $X\in\D_{\text{b}}(R)$.
	\begin{enumerate}[\rm (a)]
		\item If $\homciid_R(X)<\infty$, then for any quasi-deformation $R\xra{\varphi} R' \from Q$ such that $\id_Q(X^{\varphi})<\infty$ we have \label{160817:4a}
		\begin{align*}
			\homciid_R(X) 
				&= \id_Q(X^{\varphi}) - \pd_Q(R')\\
                &= \sup\{\depth(R_{\p}) - \width_{R_{\p}}( X_{\p}) \mid \p\in \spec[R]\}.
		\end{align*}
		\item There are inequalities \label{160817:4b}
			\[
				\operatorname{Gid}_R(X)\leq \homciid_R(X)\leq \id_R(X)
			\]
			with equality to the left of any finite value.
	\end{enumerate}
\end{thm}

\begin{proof}
	We will prove \eqref{160817:4a} and \eqref{160817:4b} simultaneously. 

	Claim 1: If $\id_R(X)<\infty$, then $\homciid_R(X) = \id_R(X) < \infty$.
	To this end, let $R\xra{\varphi} R' \from Q$ be a quasi-deformation. 
	Since $R'$ is faithfully flat over $R$, from \cite[Theorem 2.2]{christensen:imuffre} we have $\id_{R'}(X^{\varphi}) = \id_{R}(X)<\infty$. 
	We next observe that 
	\[
		\id_Q(X^{\varphi}) - \pd_Q(R') \leq \id_{R'}(X^{\varphi}) = \id_R(X) <\infty
	\]
	by \cite[Corollary 4.2(b)(I)]{avramov:hdouc}. 
	Thus, to establish the claim, it remains to verify the first inequality here is an equality.

	The Chouinard formula~\cite[Corollary 3.1]{chouinard:ofwid} and~\cite[Theorem 2.10]{yassemi:wcm} states
	\begin{equation}\label{180906:1}
            \id_{R'}(X^{\varphi}) = \sup\{\depth(R'_{\p}) - \width_{R'_{\p}}( (X^{\varphi})_{\p}) \mid \p\in \spec[R']\}.
	\end{equation}
        Let $p\in\spec[R']$ be such that $\id_{R'}(X^{\varphi}) = \depth(R'_{p}) - \width_{R'_{p}}( (X^{\varphi})_{p})$, and set $q = \tau^{-1}(p) \in \spec[Q]$. Since $\tau \colon Q \to R'$ is a codimension-$c$ deformation where $c = \pd_Q(R')$ so is $\tau_{p} \colon Q_{q} \to R'_{p}$. Thus we have $\depth(Q_{q}) = \depth(R'_{p}) + c$. This explains the second equality in the following display:
	\begin{align*}
		\id_Q(X^{\varphi})
			&\leq \id_{R'}(X^{\varphi}) + \pd_Q(R')\\
			&= \depth(R_{p}') - \width_{R'_{p}}( (X^{\varphi})_{p}) + \pd_Q(R')\\
			&= \depth(Q_{q}) - \width_{Q_{q}}( (X^{\varphi})_{q}) -\pd_Q(R') + \pd_Q(R')\\
                        &\leq \sup\{\depth(Q_{\q}) - \width_{Q_{\q}}( (X^{\varphi})_{\q}) \mid \q \in \spec[Q]\}\\
			&= \id_Q(X^{\varphi}).
	\end{align*}
	The first inequality is from the preceding paragraph, and the first equality is from our choice of $p$. 
	The second inequality is tautological and the last equality is from the Chouinard formula \eqref{180906:1}.
	This establishes Claim 1.

	For the rest of the proof, assume that $\homciid_R(X) < \infty$, and let $R\xra{\varphi} R' \from Q$ be a quasi-deformation such that $\id_Q(X^{\varphi}) < \infty$.

    	Claim 2: $\gid_R(X) = \gid_{R'}(X^{\varphi})<\infty$.
    	By \cite[Theorem A]{christensen:tgdrh} to verify Claim 2, it suffices to show that $\gid_{R'}(X^{\varphi}) < \infty$.

	Case 1: $Q$ is complete.
        Then $Q$ has a dualizing complex $D^Q$. 
	Since $\tau$ is a deformation, the complex $D^{R'} = R'\lotimes_Q D^Q$  is a dualizing for $R'$ by~\cite[(5.1)~Theorem]{avramov:lgh}. 
	Since $\id_Q(X^{\varphi}) < \infty$, we have $X^{\varphi} \in \catb_{D^Q}(Q)$ by~\cite[(3.2) Theorem]{avramov:rhafgd}. 
	Hence~\cite[(7.9)~Corollary]{avramov:rhafgd} implies $X^{\varphi}\in \mathcal{B}_{D^{R'}}(R')$, so~\cite[Theorem 4.4]{christensen:ogpifd} yields $\operatorname{Gid}_{R'}(X^{\varphi})<\infty$.

	Case 2: General case.
	There are equalities:
	\begin{align*}
		\id_Q(X^{\varphi}) 
			= \id_{\widehat{Q}}(\rhom_Q(\widehat{Q},X^{\varphi}))
			= \id_{\widehat{Q}}(\rhom_{R'}(\widehat{R'},X^{\varphi}))
	\end{align*}
        where the first equality is by~\cite[Theorem 2.2]{christensen:imuffre} and the second equality is by the isomorphism $\rhom_Q(\widehat{Q},X^{\varphi}) \simeq\rhom_{R'}(\widehat{R'},X^{\varphi})$ in $\D(\widehat{Q})$. Hence Case 1 explains the finiteness in the next display,  
	\[
		\gid_{R'}(X^{\varphi}) = \operatorname{Gid}_{\widehat{R'}}(\rhom_{R'}(\widehat{R'},X^{\varphi}))<\infty.
	\]
	The equality is from~\cite[Proposition 3.6]{christensen:tgdrh}. This establishes Claim 2.

	To complete the proof, in the following display the first equality is Chouinard's formula~\cite[Theorem 2.10]{yassemi:wcm}:
	\begin{align*}
	    \id_Q(X^{\varphi}) \!-\! \pd_Q(R')
                    &= \sup\{\depth(Q_{\q}) - \width_{Q_{\q}}( (X^{\varphi})_{\q}) \mid \q\in \spec[Q]\} - \pd_Q(R')\\
                    &= \sup\{\depth(R'_{\p}) - \width_{R'_{\p}}( (X^{\varphi})_{\p}) \mid \p\in \spec[R']\}\\
 		    &= \operatorname{Gid}_{R'}(X^{\varphi}) \\
		    &= \gid_R(X) \\
		    &= \sup\{\depth(R_{\p}) - \width_{R_{\p}}(X_{\p}) \mid \p \in \spec\}.
	\end{align*}
	The second equality is established as in the proof of Claim 1.
	The third and fifth equalities are by the version of Chouinard's formula for Gorenstein injective dimension~\cite[Theorem C]{christensen:tgdrh}.
	The fourth equality is by Claim 2.
	Notice that this sequence shows that the value $\id_Q(X^{\varphi}) - \pd_Q(R')$ is independent of the choice of quasi-deformation as long as $\id_Q(X^{\varphi}) < \infty$.
	Thus, the display justifies the remaining conclusions in the theorem.
\end{proof}

Bass' formula~\cite{bass:ugr} states that if $M$ is a non-zero finitely generated $R$-module of finite injective dimension, then $\id_R(M) = \depth(R)$.
The next result is a version of this for complexes of finite complete intersection Hom injective dimension.

\begin{cor}
\label{180724:1}
	Let $X\in \D_{\text{b}}^\text{f}(R)$ be such that $\homciid_R(X)<\infty$. Then 
	\[
		\homciid_R(X) = \depth(R) - \inf(X).
	\]
\end{cor}

\begin{proof}
	Theorem~\ref{160817:4}\eqref{160817:4b} explains the first equality in the following display.
	\[
		\homciid_R(X) = \gid_R(X) = \depth(R) - \inf(X)
	\]
	The second equality is by~\cite[Corollary 2.3]{christensen:tgdrh}.
\end{proof}

If $X\in \D_{\text{b}}(R)$ has $\id_R(X) < \infty$, then $\id_R(X) \leq -\inf(X)$ by definition.
Our next corollary is a version of this for the complete intersection Hom injective dimension.

\begin{cor}\label{prop171022h}
Let $X\in\catdb(R)$ be such that $\homciid_R(X)<\infty$. Then one has
$\homciid_R(X)\leq-\inf(X)$.
\end{cor}

\begin{proof}
The condition $\homciid_R(X)<\infty$ implies that
$$\homciid_R(X)=\operatorname{Gid}_R(X)\leq-\inf(X)$$
by Theorem~\ref{160817:4} and the definition of $\gid$.
\end{proof}

The next result is dual to \cite[Theorem F]{sather:cidfc}. It shows that one can exert a certain amount of control over the properties of quasi-deformations that detect the finiteness of $\homciid_R(X)$.

\begin{thm}
	\label{180522:1}
	Let $X\in \D_\text{b}(R)$ be such that $\homciid_R(X)<\infty$. Then there exists a quasi-deformation $R\xra{\varphi} R'\from Q$ such that $R'$ and $Q$ are complete, the closed fibre $R'/\m R'$ is artinian and Gorenstein, and $\homciid_R(X) = \id_Q(X^{\varphi}) - \pd_Q(R').$
\end{thm}

\begin{proof}
	There is a quasi-deformation $R \xra{\varphi} R' \xla{\tau} Q$ such that $\id_Q(X^{\varphi})<\infty$. 
	Let $P\in \operatorname{Min}(R'/\m R')$ and $\p = \tau^{-1}(P)$. 
	Then the diagram $R\to R'_P \from Q_{\p}$ is a quasi-deformation such that the closed fibre $R'_P/\m R'_P \cong (R'/\m R')_P$ artinian. 
	We now have the following commutative diagram:
	\begin{center}
	\begin{tikzpicture}
		\matrix[matrix of math nodes,row sep=2em, column sep=2em, text height=1.5ex, text depth=0.25ex]
		{
			& |[name=R]| Q & |[name=R2]| Q_{\p}\\
			|[name=R3]| R & |[name=S]| R' & |[name=S2]| R'_P\\
		};
		\draw[->,font=\scriptsize]
			(R) edge (R2)
			(R) edge node[left]{$\tau$} (S)
			(S) edge (S2)
			(R3) edge node[above]{$\varphi$} (S)
			(R2) edge  (S2);
	\end{tikzpicture}
	\end{center}
	where the horizontal maps are flat. Lemma~\ref{160901.3} explains the inequality ($*$) in the next display; the equality is by Hom-tensor adjointness and tensor cancellation. 
	\[
		\id_{Q_{\p}}(\rhom_{R}(R'_P,X))	= \id_{Q_{\p}}(\rhom_{R'}(R'_P,X^{\varphi})) \stackrel{(*)}{\leq} \id_Q(X^{\varphi})<\infty
	\]
	By replacing $R\to R'\from Q$ with $R \to R'_P \from Q_{\p}$ we assume that the quasi-deformation $R\to R'\from Q$ has $R'/\m R'$ artinian.

	By \cite[Lemma 3.1]{sather:cidfc} there is a commutative diagram of local ring homomorphisms
	\begin{center}
	\begin{tikzpicture}
		\matrix[matrix of math nodes,row sep=2em, column sep=2em, text height=1.5ex, text depth=0.25ex]
		{
			|[name=R]| R & |[name=R2]| R' & |[name=R3]| Q\\
				& |[name=S]| R'' & |[name=S2]| Q'\\
		};
		\draw[->,font=\scriptsize]
			(R) edge node[above]{$\varphi$} (R2)
			(R3) edge node[above]{$\tau$} (R2)
			(R) edge (S)
			(R2) edge (S)
			(S2) edge (S)
			(R3) edge  (S2);
	\end{tikzpicture}
	\end{center}
	such that $R''$ and $Q'$ are complete, and the diagram $R \to R''\from Q$ is a quasi-deformation such that $R''/ \m R''$ is Gorenstein and artinian. 

	Consider the right square of the above diagram.
	As in the earlier part of the proof, Lemma~\ref{160901.3} implies that 
	\[
		\id_{Q'}(\rhom_{R}(R'',X)) = \id_{Q'}(\rhom_{R'}(R'',X^{\varphi})) \stackrel{(\dagger)}{\leq} \id_Q(X^{\varphi})<\infty.
	\]
	The desired result now follows from Theorem~\ref{160817:4}\eqref{160817:4a}.
\end{proof}

\begin{disc}
\label{180814:1}
	There is a subtlety in the proof of Theorem~\ref{180522:1} that merits attention. Specifically, the inequalities ($*$) and ($\dagger$) do not follow directly from simple Hom-tensor adjointness arguments. A similar subtlety occurs in the proof of~\cite[Theorem F]{sather:cidfc}, though the author was unaware of it at the time; our Lemma~\ref{180607:1} addresses this.
\end{disc}

If $X\in \catdb(R)$ has $\id_R(X)<\infty$, then $X$ is in the Bass class $\catbc(R)$ with respect to every semidualizing $R$-complex $C$ by~\cite[(4.4) Proposition]{christensen:scatac};
see the end of Section~\ref{180807:1} for relevant relevant definitions.
Note that we do not know if this conclusion holds if one only assumes $\gid_{R}(X) < \infty$, unless $C$ is dualizing.
On the other hand, the following result shows that this conclusion does hold if $\homciid_R(X)<\infty$.

\begin{thm}
	\label{1805018:1}
	Let $X\in\catdb(R)$ be such that $\homciid_R(X)<\infty$. Then $X\in \catbc(R)$ for every semidualizing $R$-complex $C$.
\end{thm}

\begin{proof}
	Theorem~\ref{180522:1} provides a quasi-deformation $R \xra{\varphi} R' \from Q$ such that $Q$ is complete and the closed fibre $R'/\m R'$ is artinian and Gorenstein such that $\id_Q(X^{\varphi}) < \infty$.
	Since $R'$ is flat over $R$, the $R'$-complex $C' = R' \lotimes_R C$ is semidualizing by~\cite[(5.3) Proposition]{christensen:scatac}.
	Furthermore, since $Q$ is complete, from~\cite[Proposition~4.2]{frankild:sdcms} there exists a semidualizing $Q$ complex $B$ such that $C' \simeq R'\lotimes_Q B$. 
	As $\id_Q(X^{\varphi})<\infty$, we have $X^{\varphi}\in \cat{B}_B(Q)$.
	Apply~\cite[Proposition 5.3]{christensen:scatac} to conclude that $X^{\varphi}\in \cat{B}_{R'\lotimes_Q B}(R') = \cat{B}_{C'}(R')$ 
	and $X^{\varphi}\in \catbc(R)$. 

	Note that $\catbc(R)$ is not closed under arbitrary products. 
	However, it is straightforward to show that if $Y\in \catbc(R)$, then the product $Y^{\Lambda}$ is in $\catbc(R)$ for any (possibly infinite) index set $\Lambda$; the point here is that $Y\in\catbc(R)$ implies $Y\in \D_\text{b}(R)$ which implies in turn $Y^{\Lambda} \in \D_\text{b}(R)$. 
	In particular, we have
	\[
		\rhom_R( (R')^{(\Lambda)},X) \simeq \rhom_R( R', X)^{\Lambda} = (X^{\varphi})^{\Lambda} \in \catbc(R).
	\] 
	Since $\catbc(R)$ is closed under summands, it follows that $\rhom_R(T,X)\in \catbc(R)$ for all $T\in \operatorname{Add}(R') = T_1$, using the notation from Remark~\ref{disc171022a}. 
	Since $\catbc(R)$ satisfies the two-of-three condition, it follows by induction on $n$ that $\rhom_R(T,X)\in \catbc(R)$ for all $T\in T_n$, therefore for all $T$ in the thick subcategory generated by $\operatorname{Add}(R')$. 
	In particular, $X \simeq \rhom_R(R,X) \in \catbc(R)$ by Proposition \ref{prop171022a}.
\end{proof}

Our next result documents co-localization behavior of the complete intersection Hom injective dimension.

\begin{prop}\label{prop171022d}
Let $X\in\catdb(R)$, and let $\p\in\spec$.
Then there is an inequality $\homciid_{R_{\p}}(\Rhom{R_{\p}}{X})\leq\homciid_R(X)$.
\end{prop}

\begin{proof}
    Fact \ref{180505:1} implies that $\pd_R(R_{\p})<\infty$.
Thus, the condition $X\in\catdb(R)$ guarantees that $\Rhom{R_{\p}}X\in\catdb(R)$.

Assume without loss of generality that $\homciid_R(X)<\infty$.
By definition, there is a quasi-deformation $R\xra{\varphi} R' \xla{\tau} Q$ such that $\id_Q(X^{\varphi}) < \infty$.

Since $R'$ is faithfully flat over $R$, there is a prime ideal $\p'\in\spec[R']$ lying over $\p$. 
Set $P = \tau^{-1}(\p')\in\spec[Q]$. 
This yields the following commutative diagram of ring homomorphisms.
\begin{equation*}
\label{eq171022g}
	\begin{tikzpicture}
		\matrix[matrix of math nodes,row sep=2em, column sep=2em, text height=1.5ex, text depth=0.25ex]
		{
			|[name=R]| R & |[name=R2]| R' & |[name=R3]| Q\\
			|[name=R4]| R_{\p} & |[name=S]| R'_{\p'} & |[name=S2]| Q_{P}\\
		};
		\draw[->,font=\scriptsize]
			(R) edge node[above]{$\varphi$} (R2)
			(R3) edge node[above]{$\tau$} (R2)
			(R) edge (R4)
			(R2) edge (S)
			(S2) edge (S)
			(R3) edge (S2)
			(R4) edge (S);
	\end{tikzpicture}
\end{equation*}
Note that the rings in this diagram are local, as are the horizontal homomorphisms, but the vertical maps are not local.

Lemma~\ref{160901.3} explains the second step in the next sequence
\begin{align*}
\id_Q(X^{\varphi})
&=\id_Q(\Rhom{R'}X)\\
&\geq\id_{Q_P}(\Rhom[R']{R'_{\p'}}{\Rhom{R'}{X}})\\
&=\id_{Q_P}(\Rhom[R_{\p}]{R'_{\p'}}{\Rhom{R_{\p}}{X}}).
\end{align*}
The first step is by definition, and the third step is from Hom-tensor adjointness.
The bottom row of the above diagram is a quasi-deformation. 
So, the preceding display implies that %
$\homciid_{R_{\p}}(\Rhom{R_{\p}}{X})<\infty$.
Furthermore, Theorem~\ref{160817:4}\eqref{160817:4a} provides
\begin{align*}
\homciid_R(X)
&=\id_Q(X^{\varphi})-\pd_Q(R') \\
&\geq\id_{Q_P}(\Rhom[R_{\p}]{R'_{\p'}}{\Rhom{R_{\p}}{X}})-\pd_{Q_{\p}}(R'_{\p'})\\
&=\homciid_{R_{\p}}(\Rhom{R_{\p}}{X})
\end{align*}
as desired.
\end{proof}

At this point we do not understand the localization behaviour of the complete intersection Hom injective dimension. So we pose the following.
\begin{question}
\label{180814:2}
	If $X\in \D_{\text{b}}(R)$ and $\p\in \spec$, must we have $\homciid_{R_{\p}}(X_{\p}) \leq \homciid_R(X)$?
\end{question}

A frustrating question about the complete intersection dimension is the so called two-of-three question: if two modules in a short exact sequence have finite complete intersection dimension, must the third one as well?
(The difficulty stems from the presence of two potentially incomparable quasi-deformations.)
This question is easy to answer if one of the modules has finite projective dimension.
The analogous question for complete intersection Hom injective dimension is similarly frustrating.
The next lemma deals with a special case.
The interested reader can rotate the given triangle to obtain two minor variations of the lemma.

\begin{lem}
\label{180621:2}
	Let $A \to B \to C \to $ be an exact triangle in $\D_{\text{b}}(R)$. Assume that $\id_R(A) < \infty$.
	\begin{enumerate}[\rm (a)]
	\item Let $R\xra{\varphi} R' \from Q$ be a quasi-deformation. Then $id_Q(B^{\varphi}) < \infty$ if and only if $\id_Q(C^{\varphi})<\infty$. \label{180621:2a}
	\item $\homciid_R(B)<\infty$ if and only if $\homciid_R(C)<\infty$.\label{180621:2b}
	\end{enumerate}
\end{lem}

\begin{proof}
	(a) As $\id_R(A)<\infty$, the proof of Theorem~\ref{160817:4} Claim 1 implies that $\id_Q(A^{\varphi}) < \infty$. 
	The two-of-three condition for injective dimension now implies that $\id_Q(B^{\varphi}) < \infty$ if and only if $\id_Q(C^{\varphi}) < \infty$ as desired.

	Part (b) follows from (a) via the definition of $\homciid$.
\end{proof}

Readers may be disturbed by the fact that the complete intersection Hom injective dimension of a module seems to require 
$\rhom$ or at least a Hom-complex. 
Here we show how to avoid these in the task of detecting finiteness of the complete intersection Hom injective dimension.
The basic idea is to take a high co-syzygy.

\begin{prop}
\label{180621:1}
Let $X\in\catdb(R)$, and fix an injective resolution $X\xra\simeq I$ with $I_i=0$ for all $i>s$.
Fix an integer $j\leq\inf (X)-\dim(R)$, and set $N=\im(\partial^I_{j})$.
	\begin{enumerate}[\rm (a)]
	\item Let $R\xra{\varphi} R' \from Q$ be a quasi-deformation. \label{180621:1a}
	The following are equivalent.
	\begin{enumerate}[\rm (i)]
	\item $\id_Q(X^{\varphi}) < \infty$
	\item $\id_Q(N^{\varphi}) < \infty$
	\item $\id_Q(\hom{R'}{N}) < \infty$
	\end{enumerate}
	\item Then one has $\homciid_R(M)<\infty$ if and only if there is a quasi-deformation $R\to R'\from Q$ such that $\id_Q(\hom{R'}{N})<\infty$.\label{180621:1b}
	\end{enumerate}
\end{prop}

\begin{proof}
Consider the following ``hard truncation'' of $I$
\begin{align*}
I'&=\qquad 0\to I_s\to\cdots\to I_{j}\to 0
\end{align*}
which has $\id_R(I')<\infty$.

	\eqref{180621:1a}
	(i) $\Leftrightarrow$ (ii) Our assumptions provide an 
	exact triangle $I' \to \shift^jN \to X \to $ in $\D(R)$, so
	Lemma~\ref{180621:2}\eqref{180621:2a} implies the desired equivalence. 
	
	(ii) $\Leftrightarrow$ (iii) 
Fact~\ref{180505:1} implies $\pd_R(R')\leq \dim(R)<\infty$, so 
$$\inf(\Hom{R'}{I})=\inf(\Rhom{R'}{X})\geq \inf(X)-\dim(R)\geq j$$
by~\cite[2.4.5.~Theorem]{avramov:hdouc}. 
Since the truncation
$$0\to I_{j-1}\to I_{j-2}\to\cdots$$
is an injective resolution of $N$, it follows that $\Ext{\geq 1}{R'}{N}=0$, so
$$N^{\vf}=\Rhom{R'}{N}\simeq\Hom{R'}{N}.$$
The desired equivalence now follows.
	
	Part \eqref{180621:1b} follows from part~\eqref{180621:1a} by definition.
\end{proof}

\section{Stability Results}
\label{180807:3}

Recall that $R$ is a local ring. This section consists of stability results used in Section~\ref{180807:4} to prove Theorems~\ref{180711:2}--\ref{180711:4} from the introduction.

\begin{prop}
\label{160922.1}
	Let $X,J\in \mathcal{D}_{\text{b}}(R)$.
	Then 
        \begin{equation}
            \homciid_R(\rhom_R(X,J)) \leq \cifd_R(X) + \id_R(J)\label{160922.1.1}
        \end{equation}
        with equality holding when $J$ is a faithfully injective $R$-module, i.e., when $J$ is an injective $R$-module with $E_R(k)$ as a summand.
        In particular, $\homciid_R(X^{\vee})$ and $\cifd_R(X)$ are simultaneously finite.
\end{prop}

\begin{proof}
    Assume without loss of generality that $\id_R(J)<\infty$. Consider an arbitrary quasi-deformation $R\xra{\varphi} R'\from Q$. There are isomorphisms
	\begin{align*}
            \rhom_R(X,J)^{\varphi} 
            	&= \rhom_R(R',\rhom_R(X,J))\\
				&\simeq \rhom_R(R'\lotimes_R X,J)\\
				&\simeq \rhom_{R}(R'\lotimes_{R'}(R'\lotimes_R X),J)\\
				&\simeq \rhom_{R'}(R'\lotimes_R X,\rhom_R(R',J))\\
            	&= \rhom_{R'}(X', J^{\varphi})
	\end{align*}
        where the first and third isomorphisms are by Hom-tensor adjointness, and the second isomorphism is by tensor-cancellation. 

	For the inequality~\eqref{160922.1.1}, assume without loss of generality for this paragraph that $\cifd_R(X)= \fd_Q(X') - \pd_Q(R')<\infty$. 
        This explains the second equality in the next display; the first equality is by the preceding paragraph.
        \begin{align*}
            \id_Q(\rhom_R(X,J)^{\varphi}) - \pd_Q(R')
            &= \id_Q(\rhom_{R'}(X', J^{\varphi})) - \pd_Q(R') \\
            &\leq \fd_Q(X') + \id_{R'}(J^{\varphi}) - \pd_Q(R')\\
            &= \cifd_R(X) + \id_{R'}(J^{\varphi})\\
            &= \cifd_R(X) + \id_R(J)
        \end{align*}
	The inequality is from~\cite[Theorem 4.1(I)]{avramov:cid}, and the last equality is by the proof of Claim 1 in Theorem~\ref{160817:4}.
	This verifies inequality~\eqref{160922.1.1}.

        For the remainder of the proof assume that $J$ is a faithfully injective $R$-module. In this case $J^{\varphi}$ is a faithfully injective $R'$-module. We need to show that inequality~\eqref{160922.1.1} is an equality. To this end assume without loss of generality that $\homciid_R(\rhom_R(X,J)) = \id_Q(\rhom_R(X,J)^{\varphi}) - \pd_Q(R') <\infty$. The proof of~\cite[Theorem 4.1(I)]{avramov:cid} shows that the inequality in the preceding display is actually an equality. The desired equality now follows as in the preceding paragraph. The statement about simultaneous finiteness is a direct consequence.
\end{proof}

The goal for the remainder of the section is to prove a version of the previous result with the roles of $\homciid$ and $\cifd$ reversed.
See Theorem~\ref{170915:4}.
This requires the following preparatory results.

\begin{lem}
    \label{180506:2}
    Let $Q\to R'$ be a ring homomorphism and let $L,Z\in\catdb(R')$. Then 
    \[
        \id_Q(\rhom_{R'}(L,Z)) \leq \id_Q(Z) + \pd_{R'}(L)
    \]
    with equality holding when $L$ is a non-zero free $R'$-module.
\end{lem}

\begin{proof}
    First we prove the equality when $L$ is a non-zero free $R'$-module, i.e., when $L \simeq (R')^{(\Lambda)}$ where $\Lambda\neq \emptyset$. In this case we have
    \[
        \rhom_{R'}(L,Z) \simeq \rhom_{R'}((R')^{(\Lambda)},Z) \simeq \rhom_{R'}(R',Z)^{\Lambda} \simeq Z^{\Lambda}
    \]
    and therefore
    \[
        \id_Q(\rhom_{R'}(L,Z)) = \id_Q( Z^{\Lambda}) = \id_Q(Z)
    \]
    because $\Lambda \neq 0$. This is the desired equality in this case.

    For the inequality assume without loss of generality that $\pd_{R'}(L)<\infty$ and $\id_Q(Z)<\infty$. 
	If $L \simeq 0$, then we are done. So assume $L \not\simeq 0$ and apply a shift if necessary to assume furthermore that $\inf(L) = 0$. 
	Now one proves the inequality by induction on $\pd_{R'}(L) \geq 0$. 
	The base case is covered by the preceding paragraph and the induction step follows by a standard hard truncation argument.
\end{proof}

\begin{prop}
    \label{180506:1}
    Let $P,Y\in \catdb(R)$. Then 
    \[
        \homciid_R(\rhom_R(P,Y))\leq \homciid_R(Y) + \pd_R(P)
    \]
    with equality holding when $P$ is a non-zero free $R$-module.
\end{prop}

\begin{proof} 
        Argue as in the proof of Proposition~\ref{160922.1} using Lemma~\ref{180506:2} in place of~\cite[Theorem 4.1(I)]{avramov:cid}.
\end{proof}

Our next result complements \cite[Proposition 4.4(a)]{sather:cidfc}. It is the final key we need to prove Theorem~\ref{170915:4} below.
\begin{lem}
\label{170915:3}
Let $Y\in \catdb(R)$ be such that $\HH(Y)$ is $\m$-torsion. 
Let $K = K^R(\x)$ be the Koszul complex over $R$ on a sequence $\x = x_1,\dots,x_n \in \m$. Then 
\[
	\cifd_R(K\lotimes_R Y) = \cifd_R(Y) + n = \cifd_R(Y) + \pd_R(K).
\]
In particular, $\cifd_R(K\lotimes_R Y)$ and $\cifd_R(Y)$ are simultaneously finite.
\end{lem}

\begin{proof}
    Consider a quasi-deformation $R \xra{\varphi} R' \xla{\tau} Q$ such that the closed fibre of $\vf$ is artinian. 
    Note that~\cite[Theorem F]{sather:cidfc} shows that $\cifd_R(Y)$ is the infimum of the set of quantities $\fd_Q(Y') - \pd_Q(R')$ ranging through all such quasi-deformations; 
    and similarly for $\cifd_R(K\lotimes_R Y$).
    
    For $i = 1, \dots, n$, let $x'_i = \vf(x_i)$, and choose $y_i \in Q$ such that $\tau(y_i) = x_i'$. It is straightforward to verify that 
    \[
        (K \lotimes_R Y)' \simeq K^{R'}(\x')\lotimes_{R'} Y' \simeq K^Q(\y) \lotimes_Q Y'
    \]
    in $\catdb(Q)$. 
    Furthermore, from the assumption that $R'/ \m R'$ is artinian, it is straightforward to show that $\HH(Y')$ is $\m'$-torsion where $\m'$ is the maximal ideal of $R'$. 
    Hence $\HH(Y')$ is $\n$-torsion where $\n$ is the maximal ideal of $Q$. 
    From~\cite[Corollary 4.32]{yekutieli:hct} and~\cite[Proposition 5.4]{sather:scc} the ``small support'' 
    \[
    	\supp_Q(Y') = \{\p\in\spec[Q] \mid \kappa(\p)\lotimes_Q Y'\not\simeq 0\}
    \]
    is contained in $\{\n\}$. Thus \cite[Lemma 3.4]{sather:afcc} implies that
    \[
        \fd_Q( (K\lotimes_R Y)') = \fd_Q(K^Q(\y) \lotimes_Q Y') = \fd_Q(Y') + n.
    \]
    Subtract $\pd_Q(R')$ from the display and take an infimum to obtain the desired equality. 
    The statement about simultaneous finiteness follows directly.
\end{proof}

Next we have the main result for this section. We use it extensively in Section~\ref{180807:4}.

\begin{thm}
\label{170915:4}
    Let $X\in \catdfb(R)$. 
    \begin{enumerate}[\rm (a)]
	    \item Then $\cifd_R(X^{\vee}) \leq \homciid_R(X)$. \label{170915:4a}
	    \item If in addition $\HH_i(X)$ is Matlis reflexive (e.g., has finite length) for all $i$, then $ \cifd_R(X^{\vee})= \homciid_R(X)$; in particular, in this case $\homciid_R(X)$ and $\cifd_R(X^{\vee})$ are simultaneously finite.\label{170915:4b}
    \end{enumerate}
\end{thm}

\begin{proof}
	\eqref{170915:4b} Assume that each $\HH_i(X)$ is Matlis reflexive. 
	Then the natural biduality map $\HH_i(X) \to \HH_i(X^{\vee\vee})$ is an isomorphism for all $i$. It follows that $X \simeq X^{\vee\vee}$, so
\[
    \homciid_R(X) = \homciid_R(X^{\vee\vee}) = \cifd_R(X^{\vee})
\] 
by Proposition~\ref{160922.1}.

\eqref{170915:4a} Let $K = K^R(\x)$ be the Koszul complex where $\x = x_1, \dots, x_n \in \m$ is a generating sequence for $\m$. 
Since $X\in \D^\text{f}_{\text{b}}(R)$, the dual $X^{\vee}\in\D_\text{b}(R)$ has $\m$-torsion homology. Thus Lemma \ref{170915:3} explains the first equality in the following display.
\begin{align*}
	\cifd_R(X^{\vee})  
		&= \cifd_R(K \lotimes_R (X^{\vee})) - \pd_R(K)\\
		&= \cifd_R(\rhom_R(K,X)^{\vee}) - \pd_R(K)\\
		&= \homciid_R(\rhom_R(K,X)) - \pd_R(K)\\
		&\leq \homciid_R(X)
\end{align*}
The second equality is by the Hom-evaluation isomorphism $\rhom_R(K,X)^{\vee} \simeq K\lotimes_R (X^{\vee})$. 
The third equality is by part~\eqref{170915:4b} as $\HH(\rhom_R(K,X))$ has finite length. 
The inequality is by Proposition~\ref{180506:1}.
\end{proof}

\section{Complete Intersection and Cohen-Macaulay Properties}
\label{180807:4}

Recall that $R$ is a local ring. In this section we prove Theorems~\ref{180711:2}--\ref{180711:4} from the introduction starting with Theorem~\ref{180711:2}. Recall that the term ``formal complete intersection'' is defined in the introduction.

\begin{thm}
\label{160922.3}
        The following conditions are equivalent.
        \begin{enumerate}[\rm (i)]
            \item $\homciid_R(k)<\infty$. 
            \item $R$ is a formal complete intersection.
            \item For every $Y\in\D_{\text{b}}(R)$, we have $\homciid_R(Y)<\infty$. 
        \end{enumerate}
\end{thm}

\begin{proof}
	The implication (iii) $\Rightarrow$ (i) is trivial.

	(i) $\Rightarrow$ (ii)
	Assume that $\homciid_R(k)<\infty$. Then Theorem~\ref{170915:4}\eqref{170915:4b} explains the third equality in the next display while the second equality is from the isomorphisms $k^{\vee} \simeq \hom{k}{E_R(k)} \cong k$. 
	\[
		\cidim(k) = \cifd_R(k) = \cifd_R(k^{\vee}) = \homciid_R(k) <\infty
	\]
	By \cite[Theorem 1.3]{avramov:cid} it follows that $R$ is a formal complete intersection.
	
	(ii) $\Rightarrow$ (iii) Assume that $R$ is a formal complete intersection and $Y\in \D_{\text{b}}(R)$. 
	By defintion $\widehat{R} \cong Q / I$ where $Q$ is a regular local ring and $I$ is generated by a $Q$-regular sequence. 
	Therefore the natural maps $R \xra{\varphi} \widehat{R} \from Q$ form a quasi-deformation. 
	By Fact~\ref{180505:1} the condition $Y\in\catdb(R)$ implies $Y^{\varphi} \in \catdb(R')$,
	hence $Y^{\varphi}\in \D_\text{b}(Q)$. 
	As $Q$ is regular, we have that $\id_Q(Y^{\varphi})<\infty$ by \cite[Proposition 3.1]{avramov:hdouc}. 
	Thus $\homciid_R(Y)<\infty$.
\end{proof}

Here is the example promised before Theorem~\ref{160817:4}.

\begin{ex}
\label{180913:1}
	If $R$ is Gorenstein and not a formal complete intersection, then $\gid_R(k) = \dim(R) <\infty = \homciid_R(k)$.  On the other hand, if $R$ is a formal complete intersection but not regular, then $\homciid_R(k) = \dim(R) < \infty = \id_R(k)$.
\end{ex}

The following contains Theorem~\ref{180711:3} from the introduction.

\begin{thm}
	\label{170907:2}
	Let $X\in\D_{\text{b}}(R)$ be such that $\mu^i_R(X) < \infty$ for all $i$, e.g., $X\in\catdfb(R)$.
	If $\homciid_R(X) < \infty$, then the sequence of Bass numbers $\{\mu^i_R(X)\mid i\geq-\sup(X)\}$ is bounded above by a polynomial in $i$.
\end{thm}

\begin{proof}
	Assume $\homciid_R(X)<\infty$. 
	Theorem~\ref{170915:4}\eqref{170915:4a} implies $\cifd_R(X^{\vee})<\infty$.
	It is straightforward to show that $\beta_i^R(X^{\vee}) = \mu^i_R(X) < \infty$.
	As in the proof of \cite[Lemma 1.5]{avramov:cid} the sequence of Betti numbers $\{\beta_i^R(X^{\vee}) \mid i\geq \inf(X{^\vee}) = -\sup(X)\}$ is bounded above by a polynomial in $i$,
	so we have the desired conclusion.
\end{proof}

The following is Theoerem~\ref{180711:4} from the introduction.

\begin{thm}
\label{170915:4.2}
	Assume that $R$ a has non-zero finitely generated $R$-module $M$ such that $\homciid_R(M)<\infty$.
	Then $R$ is Cohen-Macaulay.
\end{thm}

\begin{proof}
	Theorem~\ref{170915:4}\eqref{170915:4a} implies that $\cifd_R(M^{\vee})<\infty$. 
	Therefore by~\cite[Theorem F]{sather:cidfc} there is a quasi-deformation $R\xra{\varphi} R' \from Q$ such that $R'$ and $Q$ are complete, and the closed fibre of $\varphi$ is artinian and Gorenstein, and $\fd_Q((M^{\vee})')<\infty$. 
	The restrictions on $\varphi$ imply that $E_R(k)' \simeq E_{R'}(l)$ by~\cite[Theorem 1]{foxby:imufbc}, where $l$ is the residue field of $R'$; so
	\begin{align*}
		(M^{\vee})'
			&= R'\lotimes_R \Rhom{M}{E_R(k)}\\
			&\simeq \rhom_{R'}(R'\lotimes_R M, R'\lotimes_R E_R(k))\\
			&\simeq \Rhom[R']{M'}{E_{R'}(l)}.
	\end{align*}
	Thus $\fd_Q( \Rhom[R']{M'}{E_{R'}(l)})<\infty$. 
	Notice that $M'$ is a finitely generated module over the complete local ring $R'$.
	Therefore it is Matlis reflexive over $R'$.
	From~\cite[Theorem 4.1(I)]{avramov:hdouc} we have
	\[
		\id_Q(M') = \id_Q(\Rhom[R']{\Rhom[R']{M'}{E_{R'}(l)}}{E_{R'}(l)}) < \infty.
	\]
	In conclusion $Q$ has a non-zero finitely generated module $M'$ of finite injective dimension.
	Thus $Q$ is Cohen-Macaulay by Bass' conjecture, which implies that $R'$ and $R$ are Cohen-Macaulay.
\end{proof}

\appendix

\section{Derived Functors}
\label{180725:1}

The first result in this appendix is for use in Theorem~\ref{180522:1} and Proposition~\ref{prop171022d}. 

\begin{lem}
\label{160901.3}
	Let $R,S,\widetilde{R},\widetilde{S}$ be commutative noetherian rings (not necessarily local) and consider the following commutative diagram of ring homomorphisms
	\begin{center}
	\begin{tikzpicture}
		\matrix[matrix of math nodes,row sep=2em, column sep=2em, text height=1.5ex, text depth=0.25ex]
		{
			|[name=R]| R & |[name=R2]| \widetilde{R}\\
			|[name=S]| S & |[name=S2]| \widetilde{S}\\
		};
		\draw[->, font=\scriptsize]
			(R) edge node[above]{$\beta$} (R2)
			(R) edge node[left]{$\tau$} (S)
			(S) edge node[above]{$\gamma$} (S2)
			(R2) edge node[right]{$\widetilde{\tau}$} (S2);
	\end{tikzpicture}
	\end{center}
	such that $\widetilde{S} \cong S\otimes_R \widetilde{R}$ and $\tor_{i}^R(S,\widetilde{R}) = 0$ for all $i>0$. Let $Y\in \D_{-}(S)$. Then 
        \[
            \id_{\widetilde{R}}(\rhom_{S}(\widetilde{S},Y))\leq \id_R(Y).
        \]
\end{lem}

\begin{proof}
	Replace $\widetilde{S}$ with $S \otimes_R \widetilde{R}$ if necessary to assume that $\widetilde{S} = S \otimes_R \widetilde{R}$. 
	Using~\cite[Proposition 2.1.10]{avramov:ifr} we factor $\beta$ as $R \xra{\iota} \overline{R} \xra{\varepsilon \,\, \simeq} \widetilde{R}$
	where $\overline{R}$ is a commutative differential graded (DG) $R$-algebra such that $\overline{R}_i$ is free over $R$ for all $i$, with $\overline{R}_{<0} = 0$, and $\iota$, $\varepsilon$ are DG-algebra homomorphisms such that $\varepsilon$ is a quasiisomorphism. 
	See~\cite{avramov:ifr,felix:rht} for relevant background on DG homological algebra.

	Set $\overline{S} = S\otimes_R \overline{R}$ which is a DG $S$-algebra such that $\overline{S}_i$ is free over $S$ for all $i$, with $\overline{S}_{<0} = 0$. 
	This gives the following commutative diagram of DG-algebra morphisms.
	\begin{center}
	\begin{tikzpicture}
		\matrix[matrix of math nodes,row sep=2em, column sep=2em, text height=1.5ex, text depth=0.25ex]
		{
			& & |[name=R3]| \overline{R}\\
			|[name=R]| R & & & & |[name=R2]| \widetilde{R}\\
			|[name=S]| S & & & & |[name=S2]| \widetilde{S}\\
			& & |[name=S3]| \overline{S}\\
		};
		\draw[->, font=\scriptsize]
		(R) edge node[below right]{$\beta$} (R2)
		(R) edge node[above]{$\iota$}(R3)
		(R3) edge node[above]{$\varepsilon$} node[left]{$\simeq$} (R2)
			(R) edge node[left]{$\tau$} (S)
			(S) edge node[above right]{$\gamma$} (S2)
			(S) edge node[below left]{$S\otimes_R \iota$} (S3)
			(S3) edge node[below right]{$\overline{S}\otimes_{\overline{R}} \varepsilon$} (S2)
			(R3) edge[out=250, in=110] node[left]{$\overline{\tau}$} (S3)
			(R2) edge node[right]{$\widetilde{\tau}$} (S2);
	\end{tikzpicture}
	\end{center}
	Claim: The DG-algebra morphisms $\overline{S}\otimes_{\overline{R}} \varepsilon$ and $S\otimes_R \varepsilon$ are quasiisomorphisms.
	Indeed since $\varepsilon$ is a quasiisomorphism, $\cone(\varepsilon)$ is exact. Because of the isomorphisms
	\begin{align*}
		\cone(\overline{S} \otimes_{\overline{R}} \varepsilon) &\cong \overline{S} \otimes_{\overline{R}} \cone(\varepsilon) = (S\otimes_R \overline{R}) \otimes_{\overline{R}} \cone(\varepsilon)\\
		&\cong S\otimes_R \cone(\varepsilon) \cong \cone(S \otimes_R \varepsilon)
	\end{align*}
	it suffices to show that $S \otimes_R \cone(\varepsilon)$ is exact. Observe that
	\[
		\cone(\varepsilon)\cong \cdots \to \overline{R}_2 \to \overline{R}_1 \to \overline{R}_0 \to \widetilde{R} \to 0.
	\]
	Since each $\overline{R}_i$ is free, $\tor^R_{\geq 1}(S,\overline{R}_i) = 0$. 
	Also, by assumption, we have that $\tor^R_{\geq 1}(S,\widetilde{R}) = 0$. 
	Since $\cone(\varepsilon)$ is exact, $S\otimes_R \cone(\varepsilon)$ is exact as desired.
	
	Without loss of generality assume that $\id_R(Y) < \infty$. 
	The properties of $\overline{S}$ and $\overline{R}$ listed above explain the first and last steps in the following sequence over $\overline{R}$.%
	\begin{align*}
		\rhom_S(\overline{S},Y)
			&\simeq \hom[S]{\overline{S}}{Y}\\
			&= \hom[S]{S\otimes_R \overline{R}}{Y}\\
			&\cong \hom[R]{\overline{R}}{\hom[S]{S}{Y}}\\
			&\cong \hom[R]{\overline{R}}{Y}\\
			&\simeq \rhom_R(\overline{R},Y)
	\end{align*}
	The second step is by definition, the third step is Hom-tensor adjointness, and the fourth step is induced by Hom cancellation. 
	Because $\id_R(Y) < \infty$ and $\overline{R}\in\D_\text{b}(R)$, it follows that $\rhom_S(\overline{S},Y) \simeq \rhom_R(\overline{R},Y)\in \D_\text{b}(R)$.

	The previous display gives the first isomorhpism in $\D(\overline{R})$ in the following display for each $N\in \D_\text{b}(\overline{R})$. The second and third isomorphisms are by Hom-tensor adjointness and tensor cancellation respectively.
	\begin{align}
		\rhom_{\overline{R}}(N,\rhom_S(\overline{S},Y))
			&\simeq \rhom_{\overline{R}}(N,\rhom_R(\overline{R},Y)) \notag \\
			&\simeq \rhom_R(\overline{R}\lotimes_{\overline{R}} N, Y) \notag \\
			&\simeq \rhom_R(N,Y) \label{160901.3-1}
	\end{align}
	
	Let $\rhom_S(\overline{S},Y) \xra{\simeq} I$ be a semiinjective resolution over $\overline{S}$. Then 
	\[
		\rhom_{\overline{S}}(\widetilde{S}, \rhom_S(\overline{S}, Y)) \simeq \hom[\overline{S}]{\widetilde{S}}{I}.
	\]
	Let $W$ be an $\widetilde{R}$-module.
	Let $\overline{G} \xra{\simeq} W$ be a semifree resolution over $\overline{R}$, and let $\widetilde{G} \xra{\simeq} W$ be a semifree resolution over $\widetilde{R}$ (and hence a quasiisomorphism over $\overline{R}$). Consider the following commutative diagram
	\begin{center}
	\begin{tikzpicture}
		\matrix[matrix of math nodes,row sep=2em, column sep=2em, text height=1.5ex, text depth=0.25ex]
		{
			|[name=S2]| \hom[\overline{S}]{\widetilde{S}}{I} & |[name=R2]| \hom[\overline{S}]{\overline{S}}{I}\\
			& |[name=R]| I\\
		};
		\draw[->, font=\scriptsize]
			(R2) edge node[right]{$\cong$} (R)
			(S2) edge (R)
			(S2) edge node[above]{$\simeq$} (R2);
	\end{tikzpicture}
	\end{center}
	where the vertical map is an isomorphism by Hom cancellation and the horizontal map is a quasiisomorphism because $\overline{S} \to \widetilde{S}$ is a quasiisomorphim and $I$ is semiinjective over $\overline{S}$. 
	Hence by composition $\hom[\overline{S}]{\widetilde{S}}{I} \to I$ is a quasiisomorphism.
	By~\cite[Theorem 6.10(i)]{felix:rht} we have a natural quasiisomorphism 
	\[
		\hom[\overline{R}]{\overline{G}}{I} \xra{\simeq} \hom[\widetilde{R}]{\widetilde{G}}{\hom[\overline{S}]{\widetilde{S}}{I}}.
	\]
	This explains the second isomorphism in the next display. 
	The first is by~\eqref{160901.3-1}, 
	and the third is induced by Hom-tensor adjointness and tensor cancellation.
	\begin{align*}
		\rhom_{R}(W,Y) 
			&\simeq \rhom_{\overline{R}}(W,\rhom_S(\overline{S},Y))\\
			&\simeq \rhom_{\widetilde{R}}(W,\rhom_{\overline{S}}(\widetilde{S}, \rhom_S(\overline{S},Y)))\\ 
			&\simeq \rhom_{\widetilde{R}}(W,\rhom_S(\widetilde{S},Y)) 
	\end{align*}
	This justifies the equality in the next display. 
	\[
		\inf(\rhom_{\widetilde{R}}(W,\rhom_{S}(\widetilde{S}, Y))) = \inf(\rhom_{R}(W,Y)) \geq \id_R(Y).
	\]
	The inequality is by~\cite[Theorem 2.4.I]{avramov:hdouc}.
	From this and another application \cite[Theorem 2.4.I]{avramov:hdouc} it follows that $\id_{\widetilde{R}}(\rhom_{S}(\widetilde{S},Y)) \leq \id_R(Y)$ as desired.
\end{proof}

The next result is proved like the previous one. It is not needed for the results of this paper; however, see Remark \ref{180814:1}.

\begin{lem}
	\label{180607:1}
	Let $R,S,\widetilde{R},\widetilde{S}$ be commutative noetherian rings (not necessarily local) and consider the following commutative diagram of ring homomorphisms
	\begin{center}
	\begin{tikzpicture}
		\matrix[matrix of math nodes,row sep=2em, column sep=2em, text height=1.5ex, text depth=0.25ex]
		{
			|[name=R]| R & |[name=R2]| \widetilde{R}\\
			|[name=S]| S & |[name=S2]| \widetilde{S}\\
		};
		\draw[->,font=\scriptsize]
			(R) edge (R2)
			(R) edge node[left]{$\tau$} (S)
			(S) edge (S2)
			(R2) edge node[right]{$\widetilde{\tau}$} (S2);
	\end{tikzpicture}
	\end{center}
	such that $S \cong \widetilde{R}\otimes_R \widetilde{S}$ and $\tor_{i}^R(S,\widetilde{R}) = 0$ for all $i>0$. Let $Y\in \D_{+}(S)$. Then 
        \[
            \fd_{\widetilde{R}}(\widetilde{S} \lotimes_S Y)\leq \fd_R(Y).
        \]
\end{lem}

The following result is a slight improvement on~\cite[Theorem~3.13]{mathew}. 
Our proof is similar to that of \textit{op.\ cit.}, but we include it here for the sake of completeness.

\begin{prop}\label{prop171022a}
Let $R$ be a commutative noetherian ring (not necessarily local) with $d = \dim(R) < \infty$. 
Let $R'$ be a faithfully flat $R$-algebra.
Then $R$ is in the thick subcategory $T$ of $\catd(R)$ generated by $\operatorname{Add}(R')$.
\end{prop}

\begin{proof}
Claim 1: For all projective $R$-modules $P$ and all $n\geq 1$ the tensor product $P \otimes_R (R')^{\otimes n}$ is in $T$. 
In particular, for all $n\geq 1$, we have $(R')^{\otimes n} \in T$.

Proof of Claim 1. We argue by induction on $n$.
For the base case $n=1$, note that $P$ is a summand of $R^{(B)}$ for some set $B$. By definition
we have $R^{(B)} \otimes_R R' \cong (R')^{(B)}\in\operatorname{Add}(R')$. Thus the summand $P\otimes_R R'$ is also in $\operatorname{Add}(R') \subseteq T$.

Induction step:
Assume that $n\geq 1$ and that $P\otimes_R (R')^{\otimes n} \in T$ for all $P$.
Fact~\ref{180505:1} implies that $\pd_R(R')\leq d$. 
This provides a bounded projective resolution
\begin{equation}\label{eq171022a}
0\to P_{d}\to\cdots\to P_{0}\to R'\to 0
\end{equation}
where each projective $R$-module $P_{i}$ is a summand of a free $R$-module $R^{(B_i)}$ with basis $B_i$.
Apply $\Otimes[R] -{(P\otimes_R (R')^{\otimes n})}$ to the resolution~\eqref{eq171022a}.
As $(R')^{\otimes n}$ is flat, so is $P\otimes_R (R')^{\otimes n}$. This yields an exact sequence
\[
	0\to P_d \otimes_R P \otimes_R (R')^{\otimes n} \to\cdots\to P_0 \otimes_R P\otimes (R')^{\otimes n} \to P\otimes_R (R')^{\otimes (n+1)}\to 0.
\]
Our induction hypothesis implies $P_i \otimes_R P \otimes_R (R')^{\otimes n} \in T$ for all $i$. As $T$ is thick, the above exact sequence implies that $P\otimes_R (R')^{\otimes (n+1)}\in T$. 
This establishes Claim~1.

Set $M=R'/R$, which is flat over $R$ since $R'$ is faithfully flat.
Next set $I=\shift^{-1}M$ so there is a natural exact triangle
in $\catd(R)$
\begin{equation}\label{eq171022d}
I\xra{\phi} R\to R'\to.
\end{equation}

Claim 2: For all $m,n\geq 1$ we have $\Otimes[R]{(R')^{\otimes m}}{I^{\otimes n}}\in T$.
In particular, $R'\otimes_R I^{\otimes n} \in T$ for all $n\geq 1$.

Proof of Claim 2. We argue by induction on $n$.
For the base case $n=1$, 
apply the functor $\Otimes[R]{(R')^{\otimes m}}-$ to the triangle~\eqref{eq171022d} and use the flatness of $(R')^{\otimes m}$
to get the exact triangle
$$
\Otimes[R]{(R')^{\otimes m}}I\to (R')^{\otimes m}\to(R')^{\otimes (m+1)}\to.
$$
Since $(R')^{\otimes m}$ and $(R')^{\otimes (m+1)}$ are in $T$ by Claim 1, 
so is $\Otimes[R]{(R')^{\otimes m}}I$.

The induction step is similar to the base case.
Assume that $n\geq 1$ and that $\Otimes[R]{(R')^{\otimes m}}{I^{\otimes n}}\in T$ for all $m\geq 1$.
Apply $\Otimes[R]{(\Otimes[R]{(R')^{\otimes m}}{I^{\otimes n}})}-$ to the triangle~\eqref{eq171022d} and use the flatness of $(R')^{\otimes m} \otimes_R M^{\otimes n}$
to get the exact triangle
$$
\Otimes[R]{(R')^{\otimes m}}{I^{\otimes (n+1)}}\to \Otimes[R]{(R')^{\otimes m}}{I^{\otimes n}}\to\Otimes[R]{(R')^{\otimes (m+1)}}{I^{\otimes n}}\to.
$$
Since $\Otimes[R]{(R')^{\otimes m}}{I^{\otimes n}}$ and $\Otimes[R]{(R')^{\otimes (m+1)}}{I^{\otimes n}}$ are in $T$, 
so is $\Otimes[R]{(R')^{\otimes m}}{I^{\otimes (n+1)}}$.
This establishes Claim~2.

Recall the morphism $\phi$ from~\eqref{eq171022d}. For each $n\in \bbn$, consider the natural morphism $I^{\otimes n}\xra{\phi^{\otimes n}}R^{\otimes n}\simeq R$
and the induced exact triangle
\begin{equation}\label{eq171022e}
I^{\otimes n}\xra{\phi^{\otimes n}} R\to C(n)\to.
\end{equation}

Claim 3: For all $m \geq 0$ and
all $n \geq 1$ we have $\Otimes[R]{I^{\otimes m}}{C(n)}\in T$.
In particular, $C(n)\in T$ for all $n\geq 1$.

Proof of Claim 3. We argue by induction on $n$.
For the base case $n=1$, 
compare the triangles~\eqref{eq171022d} and~\eqref{eq171022e} to conclude that $\Otimes[R]{I^{\otimes 0}} C(1) \simeq C(1) \simeq R'\in T$.
For $m\geq 1$ it follows that $\Otimes[R]{I^{\otimes m}}{C(1)}\simeq \Otimes[R]{I^{\otimes m}}{R'}\in T$ by Claim~2.

Induction step:
Assume that $n\geq 1$ and that $\Otimes[R]{I^{\otimes m}}{C(n)}\in T$ for all $m\geq 0$.
The morphism $\phi^{\otimes (n+1)}$ decomposes as the composition of the next morphisms
$$I^{\otimes (n+1)}\xra{\phi^{\otimes n}\otimes I}\Otimes[R] RI\xra\cong I\xra{\phi}R.$$
Apply $\Otimes[R] -I$ to the triangle~\eqref{eq171022e} to produce the next exact triangle
$$
I^{\otimes (n+1)}\xra{\phi^{\otimes n}\otimes I} \Otimes[R] RI\to \Otimes[R]{C(n)}I\to.
$$
The Octahedral Axiom applied to the morphisms $\phi^{\otimes n}\otimes I$ and $\phi$ (and their composition $\phi^{\otimes (n+1)}$)
yields the next exact triangle.
$$\Otimes[R]{C(n)}I
\to C(n+1)
\to C(1)
\to 
$$
Apply $\Otimes[R] -{I^{\otimes m}}$ to this triangle to obtain the next one.
$$\Otimes[R]{C(n)}{I^{\otimes (m+1)}}
\to \Otimes[R]{C(n+1)}{I^{\otimes m}}
\to \Otimes[R]{C(1)}{I^{\otimes m}}
\to 
$$
Since $\Otimes[R]{C(n)}{I^{\otimes (m+1)}},\Otimes[R]{C(1)}{I^{\otimes m}}\in T$, we have $\Otimes[R]{C(n+1)}{I^{\otimes m}}\in T$.
This establishes Claim~3.

Now we complete the proof. 
The module $M=R'/R$ is flat, hence so is $M^{\otimes (d+1)}$.
Thus, we have $\pd_R(M^{\otimes (d+1)})\leq d$ and so $\Ext{d+1}{M^{\otimes (d+1)}}R=0$.
It follows that
$$\Ext 0{I^{\otimes (d+1)}}R\cong \Ext 0{\shift^{-d-1}M^{\otimes (d+1)}}R\cong\Ext{d+1}{M^{\otimes (d+1)}}R=0.$$
It follows that the homotopy class of the
morphism $\phi^{\otimes (d+1)}$ is in $\Ext 0{I^{\otimes (d+1)}}R=0$, thus $\phi^{\otimes (d+1)}$ is nullhomotopic.
It follows that the codomain $R$ is a retract of $C(d+1)$.
Claim 3 implies that $C(d+1)$ is in $T$, which is closed under retracts. Therefore we have $R\in T$, as desired.
\end{proof}

Our point for including Proposition~\ref{prop171022a} is to obtain the next two results for use in Remark~\ref{180817:1}.

\begin{prop}\label{prop171022b}
Continue with the assumptions of Proposition~\ref{prop171022a}
and let $X\in\catd(R)$ be such that $\Rhom{R'}X\in\catd_*(R)$, where
$*\in\{+,-,b\}$.
Then for all $Z\in T$ we have $\Rhom{Z}X\in\catd_*(R)$.
\end{prop}

\begin{proof}
By Remark~\ref{disc171022a} we have $Z\in T_n$ for some $n\geq 1$. 
Argue by induction on~$n$.

Base case: $n=1$. In this case, $Z$ is a summand of $(R')^{(A)}$ for some $A$.
The condition $\Rhom{R'}X\in\catd_*(R)$ implies that
$$\Rhom{(R')^{(A)}}X\simeq\Rhom{R'}X)^{A}\in\catd_*(R).$$
It follows that the summand $\Rhom{Z}X$ of $\Rhom{(R')^{(A)}}X$ is also in $\catd_*(R)$.

Inductive step. Assume that $n\geq 1$ and  for all $Z'\in T_n$ we have $\Rhom{Z'}X\in\catd_*(R)$.
Let $Z\in T_{n+1}$. 
Then $Z$ is a retract of an object $Y\in\catd(R)$ such that there is an exact triangle 
$Y'\to Y\to Y''\to$
in $\catd(R)$ with $Y'\in T_1$ and $Y''\in T_{n}$.
Our base case and induction hypothesis imply that
$\Rhom{Y'}X,\Rhom{Y''}X\in\catd_*(R)$.
A long exact sequence argument shows that
$\Rhom{Y}X\in\catd_*(R)$.
It follows that the retract
$\Rhom{Z}X$ must also be in $\catd_*(R)$.
\end{proof}

\begin{cor}\label{cor171022a}
Continue with the assumptions of Proposition~\ref{prop171022a}.
Let $X\in\catd(R)$ be such that $\Rhom{R'}X\in\catd_*(R)$, where
$*\in\{+,-,b\}$.
Then $X\in \D_{*}(R)$.
\end{cor}

\begin{proof}
Proposition~\ref{prop171022a} implies that
$R\in T$, so
we have
$X\simeq\Rhom RX\in\catd_*(R)$
by Proposition~\ref{prop171022b}.
\end{proof}

\subsection*{Acknowledgments} 
We are grateful to Lars W. Christensen, Mohsen Gheibi, and Srikanth B. Iyengar for helpful suggestions.

\providecommand{\bysame}{\leavevmode\hbox to3em{\hrulefill}\thinspace}
\providecommand{\MR}{\relax\ifhmode\unskip\space\fi MR }
\providecommand{\MRhref}[2]{%
  \href{http://www.ams.org/mathscinet-getitem?mr=#1}{#2}
}
\providecommand{\href}[2]{#2}


\begin{thebibliography}{10}

\bibitem{auslander:hdlr}
M.~Auslander and D.~A. Buchsbaum, \emph{Homological dimension in local rings},
  Trans. Amer. Math. Soc. \textbf{85} (1957), 390--405. \MR{0086822 (19,249d)}

\bibitem{avramov:ifr}
L.~L. Avramov, \emph{Infinite free resolutions}, Six lectures on commutative
  algebra (Bellaterra, 1996), Progr. Math., vol. 166, Birkh\"auser, Basel,
  1998, pp.~1--118. \MR{99m:13022}

\bibitem{avramov:hdouc}
L.~L. Avramov and H.-B.\ Foxby, \emph{Homological dimensions of unbounded
  complexes}, J. Pure Appl. Algebra \textbf{71} (1991), 129--155.
  \MR{93g:18017}

\bibitem{avramov:lgh}
\bysame, \emph{Locally {G}orenstein homomorphisms}, Amer. J. Math. \textbf{114}
  (1992), no.~5, 1007--1047. \MR{1183530 (93i:13019)}

\bibitem{avramov:rhafgd}
\bysame, \emph{Ring homomorphisms and finite {G}orenstein dimension}, Proc.
  London Math. Soc. (3) \textbf{75} (1997), no.~2, 241--270. \MR{98d:13014}

\bibitem{avramov:cid}
L.~L. Avramov, V.~N. Gasharov, and I.~V. Peeva, \emph{Complete intersection
  dimension}, Inst. Hautes \'Etudes Sci. Publ. Math. (1997), no.~86, 67--114
  (1998). \MR{1608565 (99c:13033)}

\bibitem{avramov:holh}
L.~L. Avramov, S.\ Iyengar, and C.\ Miller, \emph{Homology over local
  homomorphisms}, Amer. J. Math. \textbf{128} (2006), no.~1, 23--90.
  \MR{2197067}

\bibitem{bass:ugr}
H.~Bass, \emph{On the ubiquity of {G}orenstein rings}, Math. Z. \textbf{82}
  (1963), 8--28. \MR{0153708 (27 \#3669)}

\bibitem{chouinard:ofwid}
L.~G. Chouinard, II, \emph{On finite weak and injective dimension}, Proc. Amer.
  Math. Soc. \textbf{60} (1976), 57--60 (1977). \MR{0417158}

\bibitem{christensen:scatac}
L.~W. Christensen, \emph{Semi-dualizing complexes and their {A}uslander
  categories}, Trans. Amer. Math. Soc. \textbf{353} (2001), no.~5, 1839--1883.
  \MR{2002a:13017}

\bibitem{christensen:btrmab}
L.~W. Christensen, H.-B. Foxby, and H.~Holm, \emph{Beyond totally reflexive
  modules and back: a survey on {G}orenstein dimensions}, Commutative
  algebra---{N}oetherian and non-{N}oetherian perspectives, Springer, New York,
  2011, pp.~101--143. \MR{2762509}

\bibitem{christensen:ogpifd}
L.~W. Christensen, A.\ Frankild, and H.\ Holm, \emph{On {G}orenstein
  projective, injective and flat dimensions---a functorial description with
  applications}, J. Algebra \textbf{302} (2006), no.~1, 231--279. \MR{2236602}

\bibitem{christensen:imuffre}
L.~W. Christensen and F.~K{\"o}ksal, \emph{Injective modules under faithfully
  flat ring extensions}, Proc. Amer. Math. Soc. \textbf{144} (2016), no.~3,
  1015--1020. \MR{3447655}

\bibitem{christensen:tgdrh}
L.~W. Christensen and S.\ Sather-Wagstaff, \emph{Transfer of {G}orenstein
  dimensions along ring homomorphisms}, J. Pure Appl. Algebra \textbf{214}
  (2010), no.~6, 982--989. \MR{2580673}

\bibitem{enochs:gipm}
E.~E. Enochs and O.~M.~G. Jenda, \emph{Gorenstein injective and projective
  modules}, Math. Z. \textbf{220} (1995), no.~4, 611--633. \MR{1363858
  (97c:16011)}

\bibitem{felix:rht}
Y.\ F{\'e}lix, S.\ Halperin, and J.-C.\ Thomas, \emph{Rational homotopy
  theory}, Graduate Texts in Mathematics, vol. 205, Springer-Verlag, New York,
  2001. \MR{1802847}

\bibitem{foxby:gmarm}
H.-B.\ Foxby, \emph{Gorenstein modules and related modules}, Math. Scand.
  \textbf{31} (1972), 267--284 (1973). \MR{48 \#6094}

\bibitem{foxby:imufbc}
\bysame, \emph{Injective modules under flat base change}, Proc. Amer. Math.
  Soc. \textbf{50} (1975), 23--27. \MR{0409439 (53 \#13194)}

\bibitem{foxby:bcfm}
\bysame, \emph{Bounded complexes of flat modules}, J. Pure Appl. Algebra
  \textbf{15} (1979), no.~2, 149--172. \MR{535182 (83c:13008)}

\bibitem{frankild:sdcms}
A.\ Frankild and S.\ Sather-Wagstaff, \emph{The set of semidualizing complexes
  is a nontrivial metric space}, J. Algebra \textbf{308} (2007), no.~1,
  124--143. \MR{2290914}

\bibitem{gelfand:moha}
S.~I. Gelfand and Y.~I. Manin, \emph{Methods of homological algebra},
  Springer-Verlag, Berlin, 1996. \MR{2003m:18001}

\bibitem{hartshorne:rad}
R.~Hartshorne, \emph{Residues and duality}, Lecture Notes in Mathematics, No.
  20, Springer-Verlag, Berlin, 1966. \MR{36 \#5145}

\bibitem{jensen:vl}
C.~U. Jensen, \emph{On the vanishing of
  {$\underset{\longleftarrow}{\lim}^{(i)}$}}, J. Algebra \textbf{15} (1970),
  151--166. \MR{0260839 (41 \#5460)}

\bibitem{mathew}
A.~Mathew, \emph{Examples of descent up to nilpotence}, preprint (2017),
  \texttt{arXiv:1701.01528v1}.

\bibitem{peskine:dpfcl}
C.\ Peskine and L.\ Szpiro, \emph{Dimension projective finie et cohomologie
  locale. {A}pplications \`a la d\'emonstration de conjectures de {M}.
  {A}uslander, {H}. {B}ass et {A}. {G}rothendieck}, Inst. Hautes \'Etudes Sci.
  Publ. Math. (1973), no.~42, 47--119. \MR{0374130 (51 \#10330)}

\bibitem{yekutieli:hct}
M.~Porta, L.~Shaul, and A.~Yekutieli, \emph{On the homology of completion and
  torsion}, Algebr. Represent. Theory \textbf{17} (2014), no.~1, 31--67.
  \MR{3160712}

\bibitem{raynaud:cpptpm}
M.\ Raynaud and L.\ Gruson, \emph{Crit\`eres de platitude et de projectivit\'e.
  {T}echniques de ``platification'' d'un module}, Invent. Math. \textbf{13}
  (1971), 1--89. \MR{0308104 (46 \#7219)}

\bibitem{roberts:ti}
P.~C. Roberts, \emph{Le th\'eor\`eme d'intersection}, C. R. Acad. Sci. Paris
  S\'er. I Math. \textbf{304} (1987), no.~7, 177--180. \MR{880574 (89b:14008)}

\bibitem{roberts:it}
\bysame, \emph{Intersection theorems}, Commutative algebra ({B}erkeley, {CA},
  1987), Math. Sci. Res. Inst. Publ., vol.~15, Springer, New York, 1989,
  pp.~417--436. \MR{1015532}

\bibitem{sahandi:hfd}
P.\ Sahandi, T.\ Sharif, and S.\ Yassemi, \emph{Homological flat dimensions},
  preprint (2007), \texttt{arXiv:0709.4078}.

\bibitem{sather:cidc}
S.~Sather-Wagstaff, \emph{Complete intersection dimensions for complexes}, J.
  Pure Appl. Algebra \textbf{190} (2004), no.~1-3, 267--290. \MR{2043332
  (2005i:13022)}

\bibitem{sather:cidfc}
\bysame, \emph{Complete intersection dimensions and {F}oxby classes}, J. Pure
  Appl. Algebra \textbf{212} (2008), no.~12, 2594--2611. \MR{2452313
  (2009h:13015)}

\bibitem{sather:afcc}
S.~Sather-Wagstaff and R.~Wicklein, \emph{Adically finite chain complexes}, J.
  Algebra Appl. \textbf{16} (2017), no.~12, 1750232, 23. \MR{3725092}

\bibitem{sather:scc}
\bysame, \emph{Support and adic finiteness for complexes}, Comm. Algebra
  \textbf{45} (2017), no.~6, 2569--2592. \MR{3594539}

\bibitem{serre:sldhdaedmn}
J.-P. Serre, \emph{Sur la dimension homologique des anneaux et des modules
  noeth\'eriens}, Proceedings of the international symposium on algebraic
  number theory, Tokyo \& Nikko, 1955 (Tokyo), Science Council of Japan, 1956,
  pp.~175--189. \MR{19,119a}

\bibitem{MR2231892}
R.~Takahashi, \emph{The existence of finitely generated modules of finite
  {G}orenstein injective dimension}, Proc. Amer. Math. Soc. \textbf{134}
  (2006), no.~11, 3115--3121. \MR{2231892}

\bibitem{yassemi:wcm}
S.~Yassemi, \emph{Width of complexes of modules}, Acta Math. Vietnam.
  \textbf{23} (1998), no.~1, 161--169. \MR{1628029 (99g:13026)}

\bibitem{MR2287566}
\bysame, \emph{A generalization of a theorem of {B}ass}, Comm. Algebra
  \textbf{35} (2007), no.~1, 249--251. \MR{2287566}

\end{thebibliography}
\end{document}